\documentclass[12pt]{article}   	% use "amsart" instead of "article" for AMSLaTeX format
\usepackage{geometry}                		% See geometry.pdf to learn the layout options. There are lots.
\usepackage[parfill]{parskip}    		% Activate to begin paragraphs with an empty line rather than an indent
\usepackage{graphicx}				% Use pdf, png, jpg, or eps§ with pdflatex; use eps in DVI mode
\usepackage[hyperindex=true,pagebackref,colorlinks=true,pdfpagemode=none,urlcolor=blue,linkcolor=blue,citecolor=blue,pdfstartview=FitH]{hyperref}
\usepackage{amsmath,amsfonts}
\usepackage{color}
\usepackage{amsthm}
\usepackage{boolexpr}
\usepackage{amssymb,latexsym,pstricks,pst-plot}
\usepackage[english]{babel}
\usepackage{pgf}
\usepackage{tikz}
\usetikzlibrary{arrows,automata}
\usepackage[latin1]{inputenc}
\usepackage[UKenglish]{isodate}% http://ctan.org/pkg/isodate
\allowdisplaybreaks

\newcommand{\BigO}[1]{\ensuremath{\operatorname{O}\bigl(#1\bigr)}}

\usepackage{amssymb}
\newtheorem{theorem}{Theorem}[section]
\newtheorem{lemma}[theorem]{Lemma}

\newtheorem{corollary}[theorem]{Corollary}

\title{Counting birthday collisions using partitions}
\author{Rob Burns and Jen McKenzie}
\begin{document}
\maketitle
\begin{abstract}
We use partitions to provide some formulae for counting s-collisions and other events in various forms of the Birthday Problem.
\end{abstract}

\section{Introduction}
\label{intro}
%\section{}
%\subsection{}
The standard Birthday Problem asks for the probability that at least two people in a group have the same birthday. Perhaps the most well known result in this area is that in a group of at least $23$ people the probability of at least two having the same birthday is around $1/2$. In the general form of the problem the number of days $n$ in the year and the size $q$ of the group are treated as variables and the outcomes are studied under various constraints on $n$ and $q$. The history of the problem is unclear. Both Richard von Mises in 1932 and Harold Davenport have been mentioned as initiators of the
problem.

Some of the questions which have been discussed in the context of the Birthday problem are:

For a year having $n$ days, what is the minimum size of a group to ensure that the probability of at least two people in the group having the same birthday is $\frac{1}{2}$? As mentioned earlier, for a year having $365$ days a minimum of $23$ people are needed in order that the probability of two people having the same birthdays is at least $1/2$. For large values of $n$ the size of the group needs to be $\BigO{\sqrt{n}}$ in order for the probability of a common birthday to be $\frac{1}{2}$. See, for example, \cite{Ahmed2000}, \cite{Brink_2012}.

What is the minimum size of a group such that the expected number of common birthdays is at least $1$? For a year containing $365$ days, a group of $28$ or more is needed before the expected number of common birthdays is greater than $1$. In general the group needs to be $\BigO{\sqrt{n}}$ in order for the expected value of a common birthday to be at least $1$.

In a group of $q$ people, what is the probability that everyone in the group shares a birthday with someone else in the group? This is known as the Strong Birthday Problem. In his survey article \cite{DASGUPTA2005377} DasGupta states that for a year having $365$ days, the group having $3064$ members is the smallest such that the probability of everyone sharing a birthday is $\frac{1}{2}$.

What can be said about the distribution of outcomes as the size $q$ of the group and the number $n$ of possible birthdays approach $\infty$? See for example \cite{Arratia:aa}, \cite{Arkhipov:aa}, \cite{DASGUPTA2005377}, \cite{HENZE1998333}.

The problem arises in a number of scenarios and lends itself to many variations. It appears in cryptography in the form of what is called the "Birthday Attack". In this situation, messages are mapped to a hash table and for security reasons it is important to know how many messages need to be hashed before two are found with the same hash value (see e.g. \cite{Su:aa}, \cite{Suzuki_2006},  \cite{Joux_2004},  \cite{Rivest_1997}).

The problem arises in the study of colourings of complete graphs (\cite{Bhattacharya:aa}, \cite{Fadnavis:aa}).

It is also related to the behaviour of certain Markov Chains (\cite{Benjamini:aa}, \cite{Kim:aa}, \cite{Montenegro:2012aa}).

In this paper we will picture the problem in terms of arranging $q$ balls inside $n$ tubes or buckets and counting various types of outcomes.

\section{Terminology}
Suppose we have $q$ numbered balls which are arranged inside $n$ numbered tubes. The tubes have the same width as the diameter of the balls so that when more than one ball is located within the same tube a column of balls forms. We denote the number of arrangements of $q$ balls into $n$ tubes by $T(n, q)$. The order of the balls within each tube is important and is taken into account when counting the number of arrangements.

We may also arrange the balls in buckets instead of tubes. For our purposes a bucket is a tube in which the order of the balls is not important. We denote the number of arrangements of $q$ balls into $n$ buckets by $B(n, q)$.

Whether dealing with tubes or buckets we will generally assume that the balls are numbered and therefore distinguishable. The case of indistinguishable balls, which are called bosons, has been studied. For example, the asymptotic behaviour of bosons as $n$, $q$  approach $\infty$ was studied in \cite{Aaronson:aa} and \cite{Arkhipov:aa}. Formulae obtained in sections \ref{tubesandballs} and \ref{bucketsandballs} can be modified to apply to bosons.

Let $s \in \mathbb{N}$ with $s \geq 2$. We say an $s$-collision has occurred when a tube (or bucket) contains $s$ or more balls. We will be counting the number of $s$-collisions which occur in an arrangement of balls. We will therefore need to define the number of $s$-collisions occurring when a tube contains $r$ balls with $r \geq s$. In the literature two separate definitions have been used to count $s$-collisions (see \cite{Arratia:aa}). Under one definition a tube containing $r \geq s$ balls contributes $r-s+1$ collisions to the count of $s$-collisions. The second definition states that a tube containing $r \geq s$ balls contributes $\binom{r}{s}$ collisions to the count of $s$-collisions in an arrangement of balls. We will use the second definition is this paper but the formulae derived here can be easily altered to accomodate the first definition.

The floor function will be denoted in the usual way by $\lfloor . \rfloor$.

\bigskip

\section{Partitions}
\label{partitiions}

We will denote a general partition of a positive integer $q$ by the letter $\lambda$. $\lambda$ is therefore a set of positive integers $$\{ \lambda_1, \lambda_2, \dots, \lambda_k \}$$ such that $$q = \sum_{i = 1}^{k} \lambda_i$$ and $$q \geq \lambda_1 \geq \lambda_2 \dots \geq \lambda_k \geq 1.$$ Here $k$ is called the size or the number of parts of the partition $\lambda$. It may also be written as  $| \lambda |$. $k$ depends on $\lambda$ but we will not generally make that explicit.

The following lemma will not be used in this paper but is provided to show that the term $\left(q!/ \prod_{i = 1}^{k} i^{ \lambda_i}\right)$ which appears in some of the subsequent formulae is an integer.

\bigskip
\begin{lemma}
\label{partitionlemma}
Let $q \in \mathbb{N}$ and $\lambda$ be a partition of $q$. Then $q!$ is divisible by $\prod_{i = 1}^{k} i^{ \lambda_i}$. In addition $q!$ is divisible by $\prod_{i = 1}^{k} \lambda_i!$.
\end{lemma}
\begin{proof}
We use the usual approach of taking an arbitrary prime $p$ and showing that the maximum power of $p$ dividing $q!$ is greater than the maximum power dividing either of $\prod_{i = 1}^{k} i^{ \lambda_i}$ or $\prod_{i = 1}^{k} \lambda_i!$. For an integer $n$ let $ord_p(n)$ denote the maximum power of $p$ dividing $n$. We know that 
$$
ord_p(q!) = \sum_{i \geq 1} \lfloor \frac{q}{p^i} \rfloor .
$$
We have
$$
ord_p \left( \prod_{i = 1}^{k} i^{ \lambda_i} \right) = \sum_{i \geq 1} \sum_{j \geq 1} \lambda_{j*p^i}.
$$
Since the $\{ \lambda_1, \lambda_2, \dots. \lambda_k \}$ forms a decreasing sequence, for fixed $i$ we have
$$
p^i * \sum_{j \geq 1}  \lambda_{j*p^i} \leq \sum_{r \geq 1} \lambda_r = q
$$
and since  $\sum_{j \geq 1}  \lambda_{j*p^i}$ is an integer it follows that
$$
\sum_{j \geq 1}  \lambda_{j*p^i} \leq \lfloor \frac{q}{p^i} \rfloor.
$$
Therefore,
$$
ord_p \left( \prod_{i = 1}^{k} i^{ \lambda_i} \right)  \leq  \sum_{i \geq 1}  \lfloor \frac{q}{p^i} \rfloor = ord_p(q!).
$$
The approach for $\prod_{i = 1}^{k} \lambda_i!$ is the same. We have
$$
ord_p(\prod_{i = 1}^{k} \lambda_i!) = \sum_{i \geq 1} \sum_{j \geq 1} \lfloor \frac{\lambda_i}{p^j} \rfloor
$$
$$
\leq \sum_{j \geq 1} \lfloor \frac{\sum_{i \geq 1} \lambda_i}{p^j} \rfloor
$$
$$
= \sum_{j \geq 1} \lfloor \frac{q}{p^j} \rfloor = ord_p(q).
$$
\end{proof}
\bigskip
The above lemma shows that the tuple $\{ q, \lambda_1, \dots , \lambda_k \}$ satisfies
$$
\frac{(q*n)!}{\prod_{i = 1}^{k} (\lambda_i * n)!} \in \mathbb{N}
$$
for every $n \in \mathbb{N}$. We say that the tuple has an integral factorial ratio. This is an area of current research. See for example recent papers by Soundararajan \cite{Soundararajan:ab}, \cite{Soundararajan:aa}.

\bigskip

\section{A commutative diagram}
\label{commuting}

Each arrangement of $q$ balls in $n$ tubes can be mapped to a partition $\lambda$ of $q$ by letting $\lambda_1$ be the number of tubes holding at least one ball, $\lambda_2$ be the number of tubes holding at least two balls etc. It is clear that $\lambda$ defined in this way satisfies the definition of a partition of $q$. We will call this mapping $\phi_T$. In the same way, each arrangement of $q$ balls in $n$ buckets can be mapped to a partition $q$ by a map which we will call $\phi_B$. Both mappings are many to one. If $q \leq n$ then the mappings are surjective, otherwise the common range of the mappings is the set of partitions $\lambda$ such that $\lambda_1 \leq n$.

Each arrangement of balls in tubes can be mapped to an arrangement of balls in buckets by ignoring the order of the balls in each tube. We will call this mapping $\phi_{TB}$. It is also many to one and surjective.

The mappings $\phi_T$, $\phi_B$ and $\phi_{TB}$ satisfy the identity
\begin{equation}
\label{phicoll}
\phi_T = \phi_B \circ \phi_{TB}.
\end{equation}
\bigskip
This identity represents the fact that the partition associated with an arrangement of balls in tubes is the same partition associated with the arrangement in buckets obtained by ignoring the order of the balls in each tube.
\bigskip

For a set $A$, we denote the number of elements in A by $| A |$.

\bigskip
\begin{lemma}
\label{mapping lemma}
Let $\lambda$ be a partition of $q$. Then we have
\begin{equation}
\label{T-1}
| \phi_T ^{-1} (\lambda) | =  q! * \binom{n}{\lambda_1}*\prod_{i=2}^{k} \binom{\lambda_{i-1}}{\lambda_i}
\end{equation}
and
\begin{equation}
\label{B-1}
| \phi_B ^{-1} (\lambda) | =  \left( q!  / \prod_{i = 1}^{k} i^{ \lambda_i}  \right) * \binom{n}{\lambda_1}*\prod_{i=2}^{k} \binom{\lambda_{i-1}}{\lambda_i}.
\end{equation}
Let $a$ be an arrangement of $q$ balls in $n$ buckets and denote the partition $\phi_B(a)$ by $\lambda$. Then
\begin{equation}
\label{TB-1}
| \phi_{TB} ^{-1} (a) | = \prod_{i = 1}^{k} i^{ \lambda_i}.
\end{equation}
\end{lemma}
\begin{proof}
For a fixed partition $\lambda$ of $q$ there are $\binom{n}{\lambda_1}$ ways of choosing the $\lambda_1$ tubes containing at least one ball, $\binom{\lambda_1}{\lambda_2}$ ways of choosing $\lambda_2$ tubes containing at least two balls etc. Therefore the number of ways that the tubes can be chosen so that the arrangement matches $\lambda$ is 
$$
\binom{n}{\lambda_1}*\prod_{i=2}^{k} \binom{\lambda_{i-1}}{\lambda_i}.
$$
For each choice of tube pattern there are $q!$ ways of arranging the $q$ balls in the tubes so that the balls match the pattern. Equation (\ref{T-1}) follows.

Let $a$ be an arrangement of $q$ balls in $n$ buckets and $\lambda = \phi_B(a)$. Denote the number of balls in the $j$-th bucket by $x_j$. By definition, for each $i \ \geq 1$,
$$
\lambda_i = | \{ j : x_j \geq i \} | .
$$
Then
$$
| \phi_{TB} ^{-1} (a) |  = \prod_{j = 1}^{n} x_j ! = \prod_{i \geq 1} i^{ | \{ j: x_j \geq i \} |} = \prod_{i = 1}^{k} i^{ \lambda_i}.
$$

Equation (\ref{B-1}) follows from equations (\ref{T-1}) and (\ref{TB-1}) and the identity (\ref{phicoll}).
\end{proof}

\bigskip

\section{Tubes and balls}
\label{tubesandballs}

In this section we present a formula for the number of arrangements of $q$ balls in $n$ tubes. 

\bigskip

\begin{theorem}
\label{tubethm1}
The number $T(n, q)$ of arrangements of  $q$ numbered balls in $n$ numbered tubes is given by the equation
\begin{equation}
\label{Tnq}
T(n, q) = q!*\sum_{\lambda : \lambda_1 \leq n} \binom{n}{\lambda_1}*\prod_{i=2}^{k} \binom{\lambda_{i-1}}{\lambda_i}.
\end{equation}
where the sum is over all partitions $\lambda$ of $q$ such that $\lambda_1 \leq n$.
\end{theorem}
\begin{proof}
Each arrangement of balls in the tubes corresponds to a partition $\lambda$ of $q$ via the mapping $\phi_T$. The number of arrangements mapped to the same $\lambda$ is given by equation (\ref{T-1}). In order to count all possible arrangements we take the sum of $ |\phi_T^{-1}(\lambda) |$ over all partitions of $q$ resulting in equation (\ref{Tnq}).
\end{proof}

\bigskip

Theorem \ref{tubethm1} can be used to obtain an expression for the number of arrangements having no $s$-collision by restricting the sum to partitions of $q$ in which $| \lambda | < s$. Fairly simple formulae result when $s = 2, 3$.

\bigskip

\begin{corollary}
\label{tubecor1}
The number of arrangements of  $q$ numbered balls in $n$ numbered tubes in which there are no $2$-collisions is 
$$
\frac{n!}{(n-q)!}
$$
\end{corollary}
\begin{proof}
In these arrangements all the balls lie on the bottom level of the tubes so the corresponding partition of $q$ must have $\lambda_2 = 0$. The only such partition is the trivial one given by $\lambda = \{ q \}$. Equation (\ref{Tnq}) then reduces to the required formula.
\end{proof}

\bigskip

\begin{corollary}
\label{tubecor2}
The number of arrangements of  $q$ numbered balls in $n$ numbered tubes in which there are no $3$-collisions is 
$$
q! * \sum_{r=0}^{r=\lfloor \frac{q}{2} \rfloor } \binom{n}{q-r} * \binom{q-r}{r}
$$
\end{corollary}
\begin{proof}
In these arrangements all the balls lie on the bottom two levels of the tubes. We therefore have $| \lambda | \leq 2$ for the corresponding partitions of $q$. Partitions satisfying this constraint are given by
$\lambda\_0 = \{ q \}$ and $\lambda\_r = \{ q-r, r \}$ for $r \in \{ 1, 2, ... , \lfloor \frac{q}{2} \rfloor \}$. The corollary follows from equation (\ref{Tnq}).
\end{proof}

\bigskip
Note that when $q \geq 2n + 1$ all arrangements have at least one $3$-collision. The expression in Corollary \ref{tubecor2} still makes sense and sums to $0$ when $q \geq 2n + 1$ taking into account the convention that $\binom{x}{y} = 0$ when $y > x$.

\bigskip

\begin{theorem}
\label{tubethm2}
Let $s \geq 2$. The total number $C_T(n, q, s)$ of $s$-collisions in all arrangements of $q$ numbered balls in $n$ numbered tubes is equal to
\begin{equation}
\label{Cnqs}
q!*\sum_{\lambda : \lambda_1 \leq n : | \lambda | \geq s} \binom{n}{\lambda_1}* \left( \prod_{i=2}^{k} \binom{\lambda_{i-1}}{\lambda_i} \right) * \left( \lambda_{k}*\binom{k}{s} + \sum_{i=s}^{k-1} \, ( \lambda_{i} - \lambda_{i+1} )*\binom{i}{s} \right)
\end{equation}
where the sum is over all partitions $\lambda$ of $q$ such that $\lambda_1 \leq n$ and $| \lambda | \geq s$.
\end{theorem}
\begin{proof}
Any arrangement corresponding to a partition of $q$ with $| \lambda | < s$ has no $s$-collisions as all balls lie below the $s$-th level of the tubes. We therefore only need to consider partitions with $| \lambda | \geq s$.  We begin with equation (\ref{Tnq}). Each term in the sum in equation (\ref{Tnq}) is the number of arrangements corresponding to a particular partition $\lambda$. The number of $s$-collisons is the same for each of arrangement having the same $\lambda$. We need to calculate the number of $s$-collisions occurring for each of these partitions. As mentioned earlier, a tube containing $r \geq s$ balls contributes $\binom{r}{s}$ $s$-collisions to the count. For the partition $\lambda$ with $k =: | \lambda | \geq s$ there are $\lambda_k$ tubes containing exactly $k$ balls, $\lambda_{k-1} - \lambda_k$ tubes containing exactly $k-1$ balls, $\dots$, $\lambda_s - \lambda_{s+1}$ tubes containing exactly $s$ balls. Therefore, each $\lambda$ with $| \lambda | \geq s$ contributes
$$
\lambda_{k}*\binom{k}{s} + \sum_{i=s}^{k-1} \, ( \lambda_{i} - \lambda_{i+1} )*\binom{i}{s}
$$
$s$-collisions to the total. Combining this with equation (\ref{Tnq}) yields equation (\ref{Cnqs}).
\end{proof}

\bigskip

\section{Buckets and balls}
\label{bucketsandballs}

In this section we replace the tubes by buckets. The results from section \ref{tubesandballs} can be used with an appropriate adjustment to take into account that the balls are not ordered within each bucket. A number of closed form expressions have been published for the number of arrangements of balls in buckets satisfying various properties. For example, McKinney (\cite{Mckinney_1966}) provided an expression for the number of arrangements in which there is no $s$-collision. Brink (\cite{Brink_2012} ) provided an exact formula for the least value of $q$ (in terms of $n$) such that the number of arrangements  containing a $2$-collision is at least a half of all arrangements. 

The number of arrangements containing an $s$-collision can also be calculated using a recursive formula. Such a formula was provided by Suzuki, Tonien, Kurosawa, and Toyota in the paper \cite{Suzuki_2006}. 

In this section we will use partitions to construct formulae for various Birthday events.

\bigskip

\begin{theorem}
\label{bucketthm1}
The number $B(n, q)$ of arrangements of  $q$ numbered balls in $n$ numbered buckets is given by the equation
\begin{equation}
\label{Bnq}
B(n, q) = \sum_{\lambda : \lambda_1 \leq n} \binom{n}{\lambda_1}*\left( \prod_{i=2}^{k} \binom{\lambda_{i-1}}{\lambda_i} \right) * \left(q!/ \prod_{i = 1}^{k} i^{ \lambda_i}\right)
\end{equation}
where the sum is over all partitions $\lambda$ of $q$ such that $\lambda_1 \leq n$.
\end{theorem}
\begin{proof}
This follows from Theorem \ref{tubethm1} and equation (\ref{TB-1}).
\end{proof}

\bigskip

Since we know that the number of arrangements of $q$ balls in $n$ buckets is $n^q$ ,we have
\begin{equation}
\label{nq}
n^q = \sum_{\lambda : \lambda_1 \leq n} \binom{n}{\lambda_1}*\left( \prod_{i=2}^{k} \binom{\lambda_{i-1}}{\lambda_i} \right) * \left(q!/ \prod_{i = 1}^{k} i^{ \lambda_i}\right).
\end{equation}
When $n = 2$, the relevant partitions of $q$ in equation (\ref{nq})  are of the form
$$
\{ 2, 2, \dots , 2, 1, 1, \dots, 1 \}.
$$
Some algebra then produces the well known formula
$$
2^q = \sum_{i = 0}^{q} \binom{q}{i}.
$$

\bigskip

\begin{corollary}
\label{bucketcor1}
The number of arrangements of  $q$ numbered balls in $n$ numbered buckets in which there are no $2$-collisions is 
$$
\frac{n!}{(n-q)!}
$$
\end{corollary}
\begin{proof}
The proof is the same as for corollary \ref{tubecor1}.
\end{proof}

\bigskip

The following corollary appears as equation (4) in DasGupta's survey article \cite{DASGUPTA2005377}.

\bigskip

\begin{corollary}
\label{bucketcor2}
The number of arrangements of  $q$ numbered balls in $n$ numbered buckets in which there are no $3$-collisions is 
$$
\sum_{r=0}^{r=\lfloor \frac{q}{2} \rfloor } \binom{n}{q-r} * \binom{q-r}{r} * \left( \frac{q!}{2^r} \right)
$$
\end{corollary}
\begin{proof}
In these arrangements each bucket contains at most $2$ balls. We therefore have $| \lambda | \leq 2$ for the corresponding partitions of $q$. Partitions satisfying this constraint are given by
$\lambda\_0 = \{ q \}$ and $\lambda\_r = \{ q-r, r \}$ for $r \in \{ 1, 2, ... , \lfloor \frac{q}{2} \rfloor \}$. For the partition $\lambda\_r$ we have 
$$
\prod_{i = 1}^{k} i^{ \lambda\_r_i} = 2^r.
$$
The corollary follows from equation (\ref{Bnq}).
\end{proof}

\bigskip

\begin{theorem}
\label{bucketthm2l}
Let $s \geq 2$. The total number $C_B(n, q, s)$ of $s$-collisions in all arrangements of $q$ numbered balls in $n$ numbered buckets is given by
\begin{equation}
\label{Bnqs}
\sum_{\lambda : \lambda_1 \leq n : | \lambda | \geq s} \binom{n}{\lambda_1}*\left( \prod_{i=2}^{k} \binom{\lambda_{i-1}}{\lambda_i} \right) * \left( \lambda_{k}*\binom{k}{s} + \sum_{i=s}^{k-1} \, ( \lambda_{i} - \lambda_{i+1} )*\binom{i}{s} \right) *\left(q!/ \prod_{i = 1}^{k} i^{ \lambda_i} \right)
\end{equation}
where the sum is over all partitions $\lambda$ of $q$ such that $\lambda_1 \leq n$ and $| \lambda | \geq s$.
\end{theorem}
\begin{proof}
This follows from Theorem \ref{tubethm2} and equation (\ref{TB-1}).
\end{proof}

\bigskip

Subsets of arrangements can be counted using equation (\ref{Bnq}) by restricting the choice of partitions in the sum. For example, to count the number of arrangements in which at least $r$ buckets have an $s$-collision, the sum in  (\ref{Bnq}) should only include partitions $\lambda$ such that $\lambda_s \geq r$. The Strong Birthday problem asks for the number of arrangements in which no bucket contains only one ball. This number is obtained from (\ref{Bnq}) by restricting to partitions $\lambda$ such that $\lambda_1 = \lambda_2$.

\bigskip

\bibliographystyle{plain}
\begin{small}
\bibliography{Birthday}
\end{small}

\end{document}